\documentclass[11pt]{amsart}
\usepackage{amsfonts, amsmath, amssymb, amsthm, stmaryrd, color, enumerate, relsize}
\usepackage[pdftex]{graphicx}
\usepackage{mathrsfs,array}
\usepackage{xy}
\usepackage{hyperref}
\usepackage{epigraph}
\usepackage{tikz-cd}
\input xy
\xyoption{all}

\numberwithin{equation}{section}
\newtheorem{theorem}{Theorem}
\numberwithin{theorem}{section}

\newtheorem{corollary}[theorem]{Corollary}
\newtheorem{proposition}[theorem]{Proposition}

\newtheorem{definition}[theorem]{Definition}

\theoremstyle{definition}
\newtheorem{remark}[theorem]{Remark}
\newtheorem{example}[theorem]{Example}

\newtheorem{warning}[theorem]{Warning}

\date{\today}

\title{Wild Betti sheaves}
\author{Peter Scholze}
\begin{document}

\begin{abstract} In this note, we consider the problem of constructing an enlargement of the category of Betti sheaves that supports an ``exponential local system'' on $\mathbb R$, and a Fourier equivalence defined on all sheaves. We show that there is a universal solution, recovering a construction of Tamarkin known also as ``enhanced sheaves''. The universality property implies that the category of coefficients of this theory is, in a suitable sense, a nontrivial $\mathbb R_{>0}$-torsor over $\mathrm{Spec}(\mathbb Z)$.
\end{abstract}

\maketitle


\tableofcontents

\section{Introduction}

There are various sheaf theories in algebraic and arithmetic geometry, with similar properties: The classical theory of Betti sheaves, the theory of $D$-modules, the theory of \'etale sheaves, the theory of arithmetic $D$-modules, etc.~. They share many common properties and in most situations one can faithfully translate from one setting into another. Even better, one can often work with the theory of motivic sheaves, which specializes to the other theories. However, there is one asymmetry that sometimes comes up: Namely, in only some of these settings, it is possible to define an interesting ``exponential local system'' $\mathrm{exp}$ on the affine line bundle $\mathbb A^1$. One way this often gets used is to define a Fourier transform on sheaves. Namely, for any vector space $V$ with dual vector space $V^\ast$, there is a natural pairing $b: V\times V^\ast\to \mathbb A^1$, and one can define the Fourier transform
\[
\mathcal F = p_{V^\ast!}(p_V^\ast\otimes b^\ast \mathrm{exp}): D(V)\to D(V^\ast).
\]
This turns out to (essentially) be its one inverse, giving an equivalence $D(V)\cong D(V^\ast)$.

The standard examples of exponential local systems are the following:
\begin{enumerate}
\item[{\rm (i)}] In the theory of $D$-modules over a field $k$ of characteristic $0$. Here, the $D$-module $\mathrm{exp}$ on $\mathbb A^1_k$ is given by the free rank $1$-module $k[T]\cdot e$ on $\mathbb A^1_k$ with connection $\nabla(f\cdot e)=(\nabla(f)-f)e$. The idea is that the basis element $e$ is given by $\mathrm{exp}(-T)$.
\item[{\rm (ii)}] In the theory of $\ell$-adic sheaves over a field $k$ of characteristic $p\neq \ell$. Here, pick an embedding $\psi: \mathbb F_p\to \overline{\mathbb Q}_\ell^\times$, and use it to turn the Artin--Schreier cover $\mathbb A^1_k\to \mathbb A^1_k: x\mapsto x^p-x$ with covering group $\mathbb F_p$ into a $\overline{\mathbb Q}_\ell$-local system $\mathrm{exp}$ on $\mathbb A^1_k$. A priori, this may not seem related to the exponential function, but note that $\psi$ is a map from an additive group to a multiplicative group, i.e.~a kind of exponential; and under Grothendieck's sheaf-function dictionary, $\psi$ yields exponential functions, and many applications to exponential sums in analytic number theory.
\end{enumerate}

The Fourier transform has many applications, for example Laumon's proof of Deligne's generalization of the Weil conjectures \cite{LaumonFourier}, and multiple applications in geometric representation theory; see \cite{LaumonICM} for an early survey. In geometric representation theory, one can often put oneself in a situation where the sheaves are scaling equivariant under the scaling action of $\mathbb G_m$ on $V$. In this case, Laumon observed that the homogeneous version of the Fourier transform can always be defined \cite{LaumonFourierHomogene}.

This discussion is closely related to the possibility of wild ramification. Namely, the exponential local system $\mathrm{exp}$ always has a ``wild'' singularity at $\infty$. In the case of $D$-modules, wild ramification means irregular singularities; while for \'etale sheaves in positive characteristic $p$, it means ramification of degree divisible by $p$. The Fourier transform will in general take tame sheaves to wild sheaves.

Thus, in sheaf theories such as Betti sheaves, or \'etale sheaves in characteristic $0$, where there are no ``wild'' sheaves, this story does not seem to exist. But in fact, in any sheaf theory it is possible to freely add an exponential local system yielding a Fourier equivalence. The goal of the present note is to make this procedure explicit in the case of Betti sheaves. This yields a category of ``wild Betti sheaves'' that contains usual Betti sheaves fully faithfully, comes with an ``exponential local system'' $\mathrm{exp}$ on $\mathbb R$, and supports a Fourier transform defined on all sheaves. In fact, in a slightly different language this theory was previously constructed by Tamarkin \cite{Tamarkin} and is closely related to the irregular Riemann--Hilbert correspondence, cf.~\cite{DAgnoloKashiwara}, where it is known as ``enhanced sheaves''. Tamarkin was motivated by applications in symplectic geometry: Namely, the coefficient category is closely related to the Novikov ring (more precisely, it has a forgetful functor to complete almost modules over the Novikov ring), and Tamarkin has shown that at least in some cases one can define the Fukaya category already with this coefficient category \cite{TamarkinFukaya}.

{\bf Acknowledgments.} It is a pleasure to dedicate this note to G\'erard Laumon. We thank Dennis Gaitsgory for comments on a preliminary version. Moreover, we are very grateful to Bingyu Zhang for kindly pointing us to the relevant previous literature on the subject.

\section{The coefficients}

The key construction we use is the category introcued below that we will use as our coefficient category. We will work $\infty$-categorically throughout and with sheaves of spectra, but of course one could specialize to abelian groups, or $R$-modules for some ring $R$.

We want to find some symmetric monoidal category $\mathcal C$ so that there is a nontrivial invertible $\mathcal C$-sheaf on $\mathbb R$. This necessarily means that $\mathcal C$ must be somewhat exotic: for any usual ring $R$, there are no nontrivial invertible sheaves of $R$-modules on $\mathbb R$, as they must be locally constant, and then trivial as $\mathbb R$ is contractible. The same argument works more generally when $\mathcal C$ has a compact unit. Indeed, in that case any invertible object is also compact. For any invertible object $\mathcal L\in \mathcal D(\mathbb R,\mathcal C)$ and $x\in \mathbb R$, let $L=\mathcal L_x\in \mathcal C$ be the fibre. By compactness of $L$, the isomorphism $L=\mathcal L_x$ spreads into a map $L\to \mathcal L|_U$ for some open neighborhood $U$ of $x$, which is also an isomorphism after shrinking $U$ (by arguing with the inverse). Finally, the contractible nature of $\mathbb R$ ensures that these local isomorphisms can be extended to a global one. Most categories $\mathcal C$ used in practice have compact unit, or at least admit a family of conservative functors to such categories, so that this argument shows that no nontrivial invertible $\mathcal C$-sheaves on $\mathbb R$ can exist. This includes all categories of quasicoherent sheaves on (analytic) stacks. It does not include, however, categories of almost modules -- here, the unit fails to be compact, and there are often no forgetful functors to categories with compact unit.

\begin{example}\label{ex:almostexp} Let $V$ be a rank $1$ valuation ring with non-discrete valuation. Let $\mathcal C$ be the category of almost modules over $V$, i.e.~the quotient of the category of $V$-modules by the full subcategory of $k$-modules, where $k$ is the residue field of $V$. Then there is a nontrivial invertible $\mathcal C$-sheaf $\mathcal L$ on $\mathbb R$ whose fibre at $r\in \mathbb R$ is the almost module corresponding to the $V$-module $\{x\in \mathrm{Frac}(V)\mid |x|\leq r\}$.
\end{example}

The author has long been intrigued by these interesting continuously varying families of line bundles on almost schemes. It turns out that the previous example yields an ``exponential local system'' defining a Fourier equivalence once one replaces $\mathcal C$ by the subcategory of complete modules. By Vaintrob's equivalence \cite{Vaintrob} recalled below, the following category is in some sense the universal case of the previous example:

\begin{definition}\label{def:wildcategory} Let $\mathbb W$ be the symmetric monoidal stable $\infty$-category of completely and continuously $\mathbb R$-filtered spectra. This is the full subcategory of the symmetric monoidal stable $\infty$-category of (ascendingly) $\mathbb R$-filtered spectra $(\mathrm{Fil}_r M)_r$, subject to:
\begin{enumerate}
\item[{\rm (i)}] Completeness: $\mathrm{lim}_{r\to -\infty} \mathrm{Fil}_r M = 0$, and
\item[{\rm (ii)}] Continuity: For all $r$, the map $\mathrm{Fil}_r M\to \mathrm{lim}_{r'>r} \mathrm{Fil}_{r'} M$ is an equivalence.
\end{enumerate}
The unit is given by $\mathbb S(0)$ with $\mathrm{Fil}_r \mathbb S(0) = 0$ for $r<0$ and $\mathrm{Fil}_r \mathbb S(0)= \mathbb S$ for $r\geq 0$. The tensor product is given by taking the tensor product of $\mathbb R$-filtered spectra and enforcing completeness and continuity.
\end{definition}

One can more precisely define the symmetric monoidal structure by regarding $\mathbb W$ as a Verdier quotient of $\mathbb R$-filtered spectra, and descending the symmetric monoidal structure.

We note that by continuity, the datum of the filtration $(\mathrm{Fil}_r M)_r$ is equivalent to the datum of the filtration $(\mathrm{Fil}_{<r} M)_r$ where $\mathrm{Fil}_{<r} M = \mathrm{colim}_{r'<r} \mathrm{Fil}_{r'} M$, subject to the continuity condition
\[
\mathrm{colim}_{r'<r} \mathrm{Fil}_{<r'} M\to \mathrm{Fil}_{<r} M
\]
being an equivalence.

There is a forgetful functor
\[
\mathbb W\to \mathrm{Sp}: (\mathrm{Fil}_r M)_r\mapsto \mathrm{colim}_{r\to \infty} \mathrm{Fil}_r M
\]
taking the underlying spectrum. The functor is only lax symmetric monoidal, and does not preserve colimits, as these operations in $\mathbb W$ must be followed by completion. There is another natural forgetful functor
\[
\mathbb W\to \mathrm{Sp}: (\mathrm{Fil}_r M)_r\mapsto \mathrm{Fil}_0 M
\]
that is right adjoint to the unit map $\mathrm{Sp}\to \mathbb W$ and hence also lax symmetric monoidal. Note that the composite $\mathrm{Sp}\to \mathbb W\to \mathrm{Sp}$ is the identity via the unit transformation, so $\mathrm{Sp}\to \mathbb W$ is fully faithful.

For any $r\in \mathbb R$, one can define an invertible object $\mathbb S(r)\in \mathbb W$ whose underlying object is $\mathbb S$, with
\[
\mathrm{Fil}_{r'} \mathbb S(r) = \mathbb S\ ,\ r'\geq r\ ,
\]
while $\mathrm{Fil}_{r'} \mathbb S(r) = 0$ for $r'<r$. The unit $\mathbb S=\mathbb S(0)$ in $\mathbb W$ is not compact, as it can be written as a colimit of $\mathbb S(r)$ over $r>0$. This colimit, computed naively in $\mathbb R$-filtered spectra, is not continuous at $0$, but enforcing continuity at $0$ yields $\mathbb S(0)$.

\begin{definition} For any topological space $X$, the symmetric monoidal $\infty$-category of wild sheaves on $X$ is
\[
\mathcal D(X,\mathbb W)\cong \mathcal D(X,\mathbb{S})\otimes \mathbb W,
\]
the presentable symmetric monoidal stable $\infty$-category of $\mathbb W$-sheaves on $X$. Concretely, this is the symmetric monoidal $\infty$-category of completely and continuously $\mathbb R$-filtered sheaves, i.e. $\mathbb R$-filtered sheaves $(\mathrm{Fil}_r \mathcal M)_r$ satisfying
\begin{enumerate}
\item[{\rm (i)}] Completeness: $\mathrm{lim}_{r\to -\infty} \mathrm{Fil}_r \mathcal M=0$.
\item[{\rm (ii)}] Continuity: For all $r$, the map $\mathrm{Fil}_r \mathcal M\to \mathrm{lim}_{r'>r} \mathrm{Fil}_{r'} \mathcal M$ is an equivalence.
\end{enumerate}
\end{definition}

\begin{remark} This construction is equivalent to a construction of Tamarkin \cite[2.2]{Tamarkin}, often denoted $\mathcal T$ in the literature; see for example \cite{APT} for a discussion of various equivalent incarnations. In the literature on the irregular Riemann--Hilbert correspondence \cite{DAgnoloKashiwara}, the construction is known as ``enhanced sheaves''.
\end{remark}

The embedding $\mathrm{Sp}\subset \mathbb W$ induces an embedding
\[
\mathcal D(X,\mathbb{S})\subset \mathcal D(X,\mathbb W)
\]
of sheaves into wild sheaves. On locally compact Hausdorff spaces, the usual formalism of six operations extends to wild sheaves, by repeating the usual constructions, for example using the general formalism developed by Heyer--Mann \cite{HeyerMann}. More precisely, we use the classes $I$ of open immersions and $P$ of proper maps, and declare $f_!$ to be left adjoint to $f^\ast$ for $f\in I$ an open immersion, and right adjoint to $f^\ast$ for $f\in P$ a proper map. The general extension results for $6$-functor formalisms, going back to Liu--Zheng and Gaitsgory--Rozenblyum, show that this gives a well-defined $6$-functor formalism. More concretely, $f_!$ is given by applying it naively on each $\mathrm{Fil}_r$, and then enforcing completeness and continuity of the resulting $\mathbb R$-filtered sheaf of spectra.

There is a functor of $(\infty,2)$-categories from the category of kernels for usual sheaves towards the category of kernels for wild sheaves, and in particular all suave or prim objects stay so in the context of wild sheaves. For example, Poincar\'e duality continues to hold, with the same dualizing objects.

\section{The exponential local system}

The family $r\mapsto \mathbb S(r)$ is continuously varying:

\begin{proposition} There is a (unique) invertible object
\[
\mathrm{exp}\in \mathcal D(\mathbb R,\mathbb W)
\]
with underlying sheaf $\mathbb S$, whose fibre at $r\in \mathbb R$ is $\mathbb S(r)$.

More generally, for any topological space $X$ and any continuous function $f: X\to \mathbb R$, there is a unique invertible object $\mathbb S(f)$ in $\mathcal D(X,\mathbb W)$ with underlying sheaf $\mathbb S$ whose fibre at $x\in X$ is $\mathbb S(f(x))$. One has $\mathbb S(f)\otimes \mathbb S(g)\cong \mathbb S(f+g)$, and in particular the inverse of $\mathbb S(f)$ is given by $\mathbb S(-f)$.
\end{proposition}

In particular, for any ``variety with potential'' $(X,f:X\to \mathbb A^1)$, we get a wild sheaf $\mathbb S(\mathrm{Re}(f))\in \mathcal D(X(\mathbb C),\mathbb W)$. This is closely related to the theory of exponential motives and rapid decay cohomology, cf.~e.g.~\cite{FresanJossen}.

\begin{proof} First, we prove uniqueness. If there are two such sheaves, then their ratio gives some invertible sheaf $\mathbb S(0)'$ with underlying sheaf $\mathbb S$, and all of whose fibres are $\mathbb S(0)$. But then each $\mathrm{Fil}_r \mathbb S(0)'$ is determined to be $\mathbb S$ or $0$ according to $r\geq 0$ or $r<0$, by comparing stalks.

For existence, we define $\mathbb S(f)$ with underlying sheaf $\mathbb S$ by its filtration
\[
\mathrm{Fil}_{<r} \mathbb S(f) = j_{r!} \mathbb S
\]
where $j_r: U_r\hookrightarrow X$ is the open immersion of
\[
U_r=\{x\in X\mid f(x)< r\}\subset X.
\]
It is clear this satisfies the continuity condition
\[
\mathrm{Fil}_{<r} \mathbb S(f) = \mathrm{colim}_{r'<r} \mathrm{Fil}_{<r'} \mathbb S(f),
\]
so setting
\[
\mathrm{Fil}_r \mathbb S(f) = \mathrm{lim}_{r'>r} \mathrm{Fil}_{<r'} \mathbb S(f)
\]
yields a completely and continuously $\mathbb R$-filtered sheaf, with the correct stalks. Concretely,
\[
\mathrm{Fil}_r \mathbb S(f) = i_{r\ast} i_r^! \mathbb S
\]
where $i_r: Z_r\hookrightarrow X$ is the closed immersion of
\[
Z_r=\{x\in X\mid f(x)\leq r\}\subset X.
\]
One gets natural maps $\mathbb S(f)\otimes \mathbb S(g)\to \mathbb S(f+g)$ from the construction, and they are isomorphisms by checking on stalks. In particular, $\mathbb S(f)$ is invertible with inverse $\mathbb S(-f)$.
\end{proof}

Concretely, both $\mathrm{Fil}_{<r} \mathrm{exp}$ and $\mathrm{Fil}_r \mathrm{exp}$ are the $!$-extension of the constant sheaf on $(-\infty,r)$.

\begin{corollary} The sheaf $\mathrm{exp}$ is additive, i.e.~the pullback of $\mathrm{exp}$ under the addition map $\mathbb R\times \mathbb R\xrightarrow{+}\mathbb R$ is $\mathrm{exp}\boxtimes \mathrm{exp}$ (uniquely in a way compatible with underlying sheaves).
\end{corollary}

\begin{proof} This follows directly from $\mathbb S(f)\otimes \mathbb S(g)\cong \mathbb S(f+g)$.
\end{proof}

We will need the following computation. This vanishing is the essential non-triviality property of the exponential local system.

\begin{proposition}\label{prop:cohomologyofexp} Under the functor of compactly supported cohomology
\[
\pi_!: \mathcal D(\mathbb R,\mathbb W)\to \mathbb W
\]
one has
\[
\pi_! \mathrm{exp} = 0.
\]
\end{proposition}

\begin{proof} For any $r\in \mathbb R$, the sheaf $\mathrm{Fil}_r \mathrm{exp}$ is the $!$-extension of the constant sheaf on $(-\infty,r)$, and hence its compactly supported cohomology is given by $\mathbb R[-1]$. But this is constant for all $r$ (with transition maps isomorphisms), so enforcing completeness to get an object of $\mathbb W$ kills it.
\end{proof}

\section{The Fourier transform}

Now let $V$ be any real vector space and let $V^\ast$ be its dual vector space. We get the pairing
\[
b: V\times V^\ast\to \mathbb R
\]
and hence the object
\[
\mathcal K_V:=b^\ast \mathrm{exp}\in \mathcal D(V\times V^\ast,\mathbb W)
\]
which one can then use to define the Fourier transform
\[
\mathcal F_V = p_{V^\ast!}(p_V^\ast\otimes \mathcal K_V): \mathcal D(V,\mathbb W)\to \mathcal D(V^\ast,\mathbb W).
\]
The following theorem is originally due to Tamarkin \cite[Theorem 3.5]{Tamarkin}. 

\begin{theorem} The functor $\mathcal F_V$ is an equivalence whose inverse is $(-1)^\ast \mathcal F_{V^\ast}[d]$.
\end{theorem}

\begin{proof} The statement is already true on the level of kernels: We have $V\xrightarrow{\mathcal K_V} V^\ast$ and $V^\ast\xrightarrow{\mathcal K_{V^\ast}} V$. Their composite is given by the kernel
\[
p_{13!}(p_{12}^\ast \mathcal K_V\otimes p_{23}^\ast \mathcal K_{V^\ast})\in \mathcal D(V\times V,\mathbb W).
\]
Using additivity of $\mathrm{exp}$, this is the pullback under $V\times V\xrightarrow{+} V$ of
\[
p_{V!} b^\ast \mathrm{exp}\in \mathcal D(V,\mathbb W).
\]
For $v\in V$, the fibre of this is given by the compactly supported cohomology of $V^\ast$ with coefficients in the pullback of $\mathrm{exp}$ under $\langle v,-\rangle: V^\ast\to \mathbb R$. If $v\neq 0$, then we can write $V^\ast=\mathbb R\oplus W$ so that $\langle v,-\rangle$ is given by projection to the first factor. Using the usual K\"unneth for compactly supported cohomology and Proposition~\ref{prop:cohomologyofexp}, we see that this compactly supported cohomology vanishes. If $v=0$, then the sheaf is the unit $\mathbb S(0)$, and the compactly supported cohomology is $\mathbb S(0)[-d]$ where $d=\mathrm{dim} V$. Thus, the composition
\[
V\xrightarrow{\mathcal K_V} V^\ast\xrightarrow{\mathcal K_{V^\ast}} V
\]
is given by the kernel encoding $(-1)^\ast[-d]$, which is an equivalence. The same is true for the composition in the other direction, yielding the claim.
\end{proof}

As usual, one can also show that $\mathcal F_V$ intertwines the convolution symmetric monoidal structure on $\mathcal D(V,\mathbb W)$ with the usual tensor product symmetric monoidal structure on $\mathcal D(V^\ast,\mathbb W)$.

Moreover, the Fourier transform is closely related to the Legendre transform.

\begin{proposition} If $V=\mathbb R$ and $f: \mathbb R\to \mathbb R$ is a convex function, with unbounded derivatives when $x\to \pm \infty$, then the Fourier transform of $\mathbb S(f)$ is given by $\mathbb S(f^\ast)[-1]$ where $f^\ast$ is the Legendre transform of $f$, taking any $y\in \mathbb R$ to the infimum of
\[
\{f(x)+xy, x\in \mathbb R\}
\]
(which is well-defined by our assumption on $f$).
\end{proposition}

Note that with our convention, $f^\ast$ is concave; the Fourier transform of $\mathbb S(f^\ast)$ will then stay in degree $0$.

\begin{proof} The value of the Fourier transform of $f$ at a point $y\in \mathbb R$ is given by the $\mathbb R$-filtered spectrum whose $<r$-th term is the compactly supported cohomology of the open subspace
\[
\{x\mid f(x) + xy < r\}\subset \mathbb R.
\]
This subset is always connected or empty by convexity, so homeomorphic to either $\mathbb R$ or $\emptyset$. Moreover, it is empty for $r$ sufficiently small (as $f$ has unbounded derivatives), and nonempty when $r$ is sufficiently large. It follows that this $\mathbb R$-filtered spectrum is of the form $\mathbb S(g)[-1]$ for some $g=g(y)$, and unraveling definitions, this is the Legendre transform.
\end{proof}

In general, when $f$ is nonconvex, the Fourier transform of $\mathbb S(f)$ is a noninvertible sheaf, and its fibre at $0$ has contributions from all the local extrema of $f$.

\section{Universality of $\mathbb W$}

We see that with the choice of $\mathbb W$ used above, things work. This raises the question whether other choices of $\mathbb W$ would have been possible. Note that we certainly need a family of invertible objects $\mathbb S(r)$ for $r\in \mathbb R$, as the fibres of the exponential local system on $\mathbb R$. However, in $\mathbb W$ we additionally have maps $\mathbb S(r)\to \mathbb S(r')$ when $r\geq r'$. These are not forced on us: Indeed, at the very least we could swap the direction of the filtration, and have such maps instead when $r\leq r'$.

We will show here that in fact up to this change in direction of the filtration, $\mathbb W$ is universal, using some ideas of Vaintrob \cite{Vaintrob}. In the process, we get a different presentation of $\mathbb W$, showcasing its relation to the Novikov ring. Related ideas have also recently been expressed in \cite[Section 4.4]{Efimov}, \cite{APT}.

For any presentable symmetric monoidal $\infty$-category $\mathcal C$, we might hope to find an invertible
\[
\mathcal L\in \mathcal D(\mathbb R,\mathcal C) = \mathcal D(\mathbb R,\mathbb{S})\otimes \mathcal C
\]
that is moreover equipped with an isomorphism $\mathrm{add}^\ast L\cong L\boxtimes L$ and suitable coherence data. (This does, in fact, guarantee invertibility, if one adds the unit isomorphism $0^\ast L\cong 1$.) The datum of $\mathcal L$ is in fact equivalent to the datum of a symmetric monoidal functor
\[
(\mathcal D(\mathbb R,\mathbb{S}),\star)\to \mathcal C
\]
where the source is endowed with the convolution product $\star$.

\begin{proposition}[{\cite[variant of Theorem 2]{Vaintrob}}] There is an equivalence of symmetric monoidal $\infty$-categories
\[
(\mathcal D(\mathbb R,\mathbb{S}),\star)\cong (\mathcal D_{\mathrm{qc}}(\tilde{\mathbb P}^{1,a}/\tilde{\mathbb G}_m),\otimes)
\]
where
\[
\tilde{\mathbb P}^1 = \varprojlim_{n,x\mapsto x^n} \mathbb P^1
\]
is the infinite root version of $\mathbb P^1$, with its action of the infinite root version
\[
\tilde{\mathbb G}_m = \varprojlim_{n,x\mapsto x^n} \mathbb G_m
\]
of the multiplicative group, and $\tilde{\mathbb P}^{1,a}$ is the almost scheme, with respect to the almost structure at $0$ and $\infty$ given by the infinite roots.
\end{proposition}

\begin{proof}[Sketch] This can be glued from several equivalences. In fact, quasicoherent sheaves on $\tilde{\mathbb A}^{1,a}/\tilde{\mathbb G}_m$ are equivalent to sheaves on $\mathbb R$ with nonnegative singular support (cf.~below), and similarly for the complementary $\tilde{\mathbb A}^{1,a}$ and nonpositive singular support. The two charts meet in $\tilde{\mathbb G}_m/\tilde{\mathbb G}_m=\ast$, and quasicoherent sheaves are just spectra, which are also equivalent to local systems of spectra on $\mathbb R$.

For the equivalence of quasicoherent sheaves on $\tilde{\mathbb A}^{1,a}/\tilde{\mathbb G}_m$ with sheaves on $\mathbb R$ with nonnegative singular support, one first recalls that quasicoherent sheaves on $\tilde{\mathbb A}^1/\tilde{\mathbb G}_m$ are equivalent to $\mathbb Q$-filtered spectra. Passing to the almost category gives the equivalence with continuously $\mathbb R$-filtered spectra (noting that continuity lets one pass from $\mathbb Q$- to $\mathbb R$-indexed filtrations uniquely). Now $\mathbb R$-filtered spectra are equivalent to sheaves on $\mathbb R$ with nonnegative singular support by taking a sheaf $\mathcal F$ to the system of $R\Gamma_c((-\infty,r),\mathcal F)$. 
\end{proof}

Thus, possible candidates for an exponential local system $\mathcal L$ with coefficients in $\mathcal C$ are parametrized by symmetric monoidal functors
\[
\mathcal D_{\mathrm{qc}}(\tilde{\mathbb P}^{1,a}/\tilde{\mathbb G}_m)\to \mathcal C.
\]
In order for this to induce a Fourier equivalence, one needs the vanishing of Proposition~\ref{prop:cohomologyofexp}, and this means in particular that $\mathcal L$ can nowhere be trivial. Thus, $\mathcal C$ must be trivial after restriction to
\[
\ast=\tilde{\mathbb G}_m/\tilde{\mathbb G}_m\subset \tilde{\mathbb P}^{1,a}/\tilde{\mathbb G}_m.
\]
But then $\mathcal C$ becomes linear over the resulting quotient category
\[
\mathcal D_{\mathrm{qc}}(\tilde{\mathbb P}^{1,a}/\tilde{\mathbb G}_m) / \mathcal D_{\mathrm{qc}}(\tilde{\mathbb G}_m/\tilde{\mathbb G}_m).
\]
But this decomposes into a product of two categories, one at $0$ and one at $\infty$. The category at $0$ is precisely the category $\mathbb W$ of completely and continuously $\mathbb R$-filtered spectra! The category at $\infty$ is the same, up to reversing filtrations. This proves universality of $\mathbb W$.

At the same time, this discussion identifies $\mathbb W$ with the subcategory of complete (at $0$) modules on the almost stack
\[
\tilde{\mathbb A}^{1,a}/\tilde{\mathbb G}_m.
\]
In this sense, the category $\mathbb W$ is the universal case of the construction of Example~\ref{ex:almostexp}. In particular, forgetting the $\tilde{\mathbb G}_m$-equivariance, this gives a symmetric monoidal functor from $\mathbb W$ to $T$-complete almost modules over $\mathbb S[[T^{1/\infty}]]$. This is a version of the Novikov ring. In particular, by base change, the results of this paper also apply to sheaves with coefficients in complete almost modules over the Novikov ring, or complete almost modules over any rank $1$ valuation ring.

\section{A canonical $\mathbb R_{>0}$-torsor over $\mathrm{Spec}(\mathbb Z)$}

In fact, this discussion proves more. Namely, if one extends the framework of algebraic geometry to encompass instead of commutative rings all presentable symmetric monoidal stable $\infty$-categories, then it turns out
\[
\mathrm{Spec}(\mathbb W)\to \mathrm{Spec}(\mathbb S)
\]
is a torsor under $\mathbb R_{>0,\mathrm{Betti}} := \mathrm{Spec}(\mathcal D(\mathbb R_{>0},\mathbb S))$. Namely, there is an action of $\mathbb R_{>0}$ on $\mathbb W$ via multiplicative rescaling of the filtration, encoded in a coaction
\[
\mathbb W\to \mathbb W\otimes \mathcal D(\mathbb R_{>0},\mathbb S) = \mathcal D(\mathbb R_{>0},\mathbb W)
\]
sending $(\mathrm{Fil}_r M)_r$ to the sheaf with stalk at $t\in \mathbb R_{>0}$ given by $(\mathrm{Fil}_{rt} M)_r$.

\begin{proposition} The preceding functor induces an equivalence
\[
\mathbb W\otimes_{\mathbb S}\mathbb W\cong \mathcal D(\mathbb R_{>0},\mathbb W).
\]
\end{proposition}

Geometrically, this is an isomorphism
\[
\mathbb R_{>0,\mathrm{Betti}}\times_{\mathrm{Spec}(\mathbb S)} \mathrm{Spec}(\mathbb W)\cong \mathrm{Spec}(\mathbb W)\times_{\mathrm{Spec}(\mathbb S)} \mathrm{Spec}(\mathbb W).
\]

This equivalence can be proved directly, but it follows from the universal property of $\mathbb W$. Namely, the Fourier equivalence yields a $\mathbb W$-linear symmetric monoidal equivalence
\[
(\mathcal D(\mathbb R,\mathbb W),\star)\cong (\mathcal D(\mathbb R,\mathbb W),\otimes)
\]
which, when combined with the base change of Vaintrob's equivalence to $\mathbb W$, yields a $\mathbb W$-linear symmetric monoidal equivalence
\[
(\mathcal D_{\mathrm{qc}}(\tilde{\mathbb P}^{1,a}/\tilde{\mathbb G}_m),\otimes)\otimes_{\mathbb S} \mathbb W\cong (\mathcal D(\mathbb R,\mathbb W),\otimes).
\]
Under this equivalence, $\tilde{\mathbb G}_m/\tilde{\mathbb G}_m\subset \tilde{\mathbb P}^{1,a}/\tilde{\mathbb G}_m$ corresponds to $\{0\}\subset \mathbb R$. Removing this part, both sides decompose into two pieces (sheaves near $0$/$\infty$ resp.~sheaves on $\mathbb R_{>0}$/$\mathbb R_{<0}$), and we get the desired equivalence
\[
\mathbb W\otimes_{\mathbb S} \mathbb W\cong \mathcal D(\mathbb R_{>0},\mathbb W)
\]
that can be unraveled to be the construction above.

\begin{warning} There are two interpretations of $\mathbb W$ now: One as sheaves on $\mathbb R$ with nonnegative singular support and a certain completeness condition (endowed with the convolution symmetric monoidal structure), and another as a twisted form of sheaves on $\mathbb R_{>0}$ (endowed with the usual tensor product of sheaves). These two interpretations are quite different. After base change to $\mathbb W$, they are Fourier dual under the Fourier equivalence on $\mathbb R$: Namely, nonnegative singular support is swapped with physical support in $\mathbb R_{\geq 0}$ under the Fourier equivalence, and the completeness condition amounts to taking the quotient by sheaves supported at $\{0\}\subset \mathbb R_{\geq 0}$.
\end{warning}

The perspective that $\mathbb W$ is a twisted form of $\mathcal D(\mathbb R_{>0},\mathbb S)$ explains some features of the situation. First, it explains that the unit of $\mathbb W$ is not compact (as $\mathbb R_{>0}$ is noncompact). It also explains the fully faithfulness of $\mathrm{Sp}\to \mathbb W$ (as $\mathbb R_{>0}$ is contractible). Indeed, both of these statements can be proved by descent along $\mathrm{Spec}(\mathbb W)\to \mathrm{Spec}(\mathbb S)$, and then follow from the paranthetical statements. It also says that wild sheaves on $X$ are a twisted form of sheaves on $X\times \mathbb R_{>0}$, again in the sense that after pullback along $\mathrm{Spec}(\mathbb W)\to \mathrm{Spec}(\mathbb S)$, they become equivalent.

The existence of a canonical nontrivial $\mathbb R_{>0,\mathrm{Betti}}$-torsor over $\mathrm{Spec}(\mathbb S)$ (thus, over $\mathrm{Spec}(\mathbb Z)$) is quite surprising. It gives a new general direction to localize in, in algebraic and arithmetic geometry. For example, $\mathrm{Spec}(\mathbb W)$ has something like a norm in the sense developed by Clausen and the author (as a multiplicative ``norm'' map $|\cdot|: \mathbb P^1_{\mathbb W}\to [0,\infty]_{\mathrm{Betti}}$ with some properties; however, here, the norm is defined only on an infinitely ramified version $\tilde{\mathbb P}^1$ of $\mathbb P^1$). This still implies that for any (affine, for simplicity) perfect $\mathbb F_p$-scheme $X=\mathrm{Spec}(A)$, its base change to $\mathrm{Spec}(\mathbb W)$ has a canonical map $X_{\mathbb W}\to \mathcal M(A)$ to the Berkovich space of $A$. After base change to $\mathrm{Spec}(\mathbb W)$, one can localize in algebraic geometry not just over spectral spaces like the topological space $\mathrm{Spec}(A)$, but over compact Hausdorff spaces!

\section{Outlook}

The construction works much more generally. Assume you have any sheaf theory encoded in a ring stack $\mathcal R$ over some base $\mathrm{Spec}(A)$. One can then parametrize additive line bundles on $\mathcal R$ which induce a Fourier equivalence. Working in a suitable general kind of algebraic geometry, as in the previous section, this always defines over $\mathrm{Spec}(A)$ a quasi-torsor under the group of units $\mathcal R^\times$. Indeed, $\mathcal R^\times$ acts on choices of additive line bundles via pullback under the multiplication on $\mathcal R$. Once one has such an additive line bundle, this action is simply transitive -- this follows from the same argument as in the previous section.

This paper deals with the case of the ring stack $\mathbb R_{\mathrm{Betti}}$ over $\mathrm{Spec}(\mathbb S)$; already this case yields highly nontrivial structures.

This way, any kind of sheaf theory encoded in a ring stack, yielding a realization of usual motives, can be upgraded to a sheaf theory which yields a realization of exponential motives (taking a variety with potential $(X,f: X\to \mathbb A^1)$ to the cohomology of $f^\ast \mathrm{exp}$).

\bibliographystyle{amsalpha}
\bibliography{WildBetti}

\end{document}